\documentclass[10pt,twoside]{siamltex}
\usepackage[english]{babel}
\usepackage{amsfonts,epsfig}
\usepackage{amsmath}
\usepackage{hyperref}

\newtheorem{remark}[theorem]{Remark}

\begin{document}

\bibliographystyle{plain}
\title{
The behaviour of the complete eigenstructure of a polynomial matrix under a generic rational transformation.
}

\author{
Vanni\ Noferini\thanks{Department of Mathematics,
University of Pisa, Largo Bruno Pontecorvo 5, 56127 Pisa, Italy.
(noferini@mail.dm.unipi.it).}}

\pagestyle{myheadings}
\markboth{V.\ Noferini}{Behaviour of polynomial matrix eigenstructures under generic rational transformations}
\maketitle

\begin{abstract}
Given a polynomial matrix $P(x)$ of grade $g$ and a rational function $x(y)=n(y)/d(y)$, where $n(y)$ and $d(y)$ are coprime nonzero scalar polynomials, the polynomial matrix $Q(y):=[d(y)]^g P(x(y))$ is defined. The complete eigenstructures of $P(x)$ and $Q(y)$ are related, including characteristic values, elementary divisors and minimal indices. A Theorem on the matter, valid in the most general hypotheses, is stated and proved.


\end{abstract}


\section{Introduction}
A polynomial matrix is a matrix whose entries belong to some polynomial ring $R$ \cite{tom}. In this paper we will always assume that 
$R$ is a principal ideal domain. This condition is equivalent to $R=\mathbb{F}[x]$, the ring of univariate polynomials in $x$ with coefficients lying in some field $\mathbb{F}$.

An important property of a polynomial matrix with entries in $\mathbb{F}[x]$ is its complete eigenstructure, whose definition is given in Subsection \ref{completeigenstructure}. The name \emph{eigenstructure} comes from the special case where $\mathbb{F}=\mathbb{C}$; in this context, polynomial matrices are usually seen instead as matrix polynomials, that is polynomials whose coefficients are matrices \cite{bible}. Any matrix polynomial is associated with a polynomial eigenvalue problem (PEP); the complete eigenstructure is strictly related with the properties of the associated PEP. More precisely, it gives the complete information about the eigenvalues, eigenvectors and Jordan chains of the matrix polynomial, and also about the Kronecker form of any strong linearization of the matrix polynomial \cite{dopico}. Polynomial eigenvalue problems arise in many applications, from mathematics, science and engineering; both their algebraic properties and the numerical methods for their approximate solutions are widely studied. See, e.g., \cite{golub, tisseur, mehrmannvoss}.

The aim of this paper is to investigate the link between the complete eigenstructures of two polynomial matrices $P(x)$ and $Q(y)$ 
related one to another by a rational transformation $x(y)$ of the variable. In order to better explain the question we are interested 
in, let us consider the following example, where $R=\mathbb{C}[x]$. Suppose that we have to deal with the polynomial matrix \[P(x)=
\left[\begin{array}{ccc} x^2-20x & 0 & 0\\x-20 &
x^2-20x & 0\\0 & 0 & x\\0 & 0 & x^2\\0 & 0 & 0\end{array}\right];
\]
if we choose grade($P(x)$)$=2$ (the grade of $P(x)$ is an arbitrary integer $g$ such that $g \geq \deg P(x)$; more details are given in Section \ref{definiz}), then the complete eigenstructure of $P(x)$ is the following:
\begin{itemize}
\item the elementary divisors of $P(x)$ are $(x-20)$, $x$, $(x-20)$, $x^2$;
\item there are no right minimal indices;
\item the left minimal indices of $P(x)$ are $0, \ 1$.
\end{itemize}
The rational change of variable $x(y)=\frac{16 y^2 - 25}{y^2-y}$ induces an application $\Phi_2$, as defined in \eqref{rationalfunction}, such that $\Phi_2(P(x))=(y^2-y)^2 P(\frac{16 y^2 - 25}{y^2-y})=:Q(y)$, with grade($Q(y))=4$ (see Section \ref{rationalfun}) and
\[Q(y)=
\left[\begin{array}{cccc} (25-16y^2)(2y-5)^2 & 0 & 0\\(y-y^2)(2y-5)^2 &
(25-16y^2)(2y-5)^2 & 0\\0 & 0 & (y^2-y)(16y^2-25)\\0 & 0 & (16y^2-25)^2\\0 & 0 & 0\end{array}\right].
\]
By studying the complete eigenstructure of $Q(y)$ we find out that
\begin{itemize}
\item the elementary divisors of $Q(y)$ are $(y-\frac{5}{2})^2$, $(y-\frac{5}{4})$, $(y+\frac{5}{4})$, $(y-\frac{5}{2})^2$, $(y-\frac{5}{4})^2$, $(y+\frac{5}{4})^2$;
\item there are no right minimal indices;
\item the left minimal indices of $Q(y)$ are $0, \ 2$.
\end{itemize}
Notice that $x(\frac{5}{2})=20$, $x(\pm \frac{5}{4})=0$, and that $y=\frac{5}{2}$ is a root of multiplicity $2$ of the equation 
$x(y)=20$ while $y=\pm \frac{5}{4}$ are roots of multiplicity $1$ of the equation $x(y)=0$. We can therefore conjecture that if 
$(x-x_0)^{\ell}$ is an elementary divisor of $P(x)$ and $y_0$ is a root of multiplicity $m$ of the equation 
$x(y)=x_0$ then $(y-y_0)^{m \cdot \ell}$ is an elementary divisor of $Q(y)$. Moreover, we see that apparently the minimal indices have 
been multiplied by a factor $2$; notice that $2$ is the degree of the considered rational transformation (that is the maximum of the 
degrees of the numerator and the denominator). 

The main result of the present paper is the proof that the conjectures above, which will be stated more precisely in Section \ref{themainresult}, are true for every rational transformation of the variable $x(y)$ and every polynomial matrix $P(x)$. Moreover, analogous properties hold for infinite elementary divisors and right minimal indices.

The motivation for this work comes from the will to generalise the partial results derived in \cite{dicksonpal}, where we considered the particular case of a square and regular polynomial matrix with entries in $\mathbb{C}[x]$ and without infinite elementary divisors, and the Dickson change of variable $x(y)=\frac{y^2+1}{y}$. Moreover, we wish to extend the results by D. S. Mackey and N. Mackey \cite{iciam}, who described the special case of rational transformations of degree $1$, also known as M\"{o}bius transformations. The present contribution is offered as both a synthesis and an extension of the previous works cited above.
 
The results provided in this paper can be used to design numerical methods for the approximate solution of PEPs. An example in this regard, restricted to the case of the Dickson transformation, is given in \cite{dicksonpal} for the solution of the palindromic PEP. 
 
The structure of this paper is the following: in Section \ref{definiz} we expose the theoretical background we are going to work within, 
and we give some basic definitions that we will use later on. In Section \ref{rationalfun} we formally define the application between 
polynomial matrices induced by a rational change of variable and we present some intermediate results. Our main result is Theorem 
\ref{main}, which is stated and commented in Section \ref{themainresult}; Sections \ref{proof1} and \ref{proof2} are devoted to the 
proof of our result. 
 For the sake of simplicity, in Sections \ref{proof1} and \ref{proof2} we assume that the underlying field is 
algebraically closed: in Section \ref{comments} we show how the result still holds for an arbitrary field. 
Finally, in Section 
\ref{app}, root polynomials are introduced in order to prove a technical Lemma.

The first part of Theorem 
\ref{main} was stated and proved, but only for a very special case, in \cite{dicksonpal}. Besides the generalisation to a generic 
rational transformation and a generic polynomial matrix, this paper also contains the analysis of what happens to minimal indices 
and infinite elementary divisors.

\section{Preliminary definitions}\label{definiz}
In this Section we describe our notation and recall some basic definitions. 

\subsection{Basic facts on polynomials}\label{gradi}

Let $Z$ be a ring and let $Z[x]$ be the ring of the univariate polynomials in the variable $x$ with coefficients in $Z$.
We denote the degree of $z \in Z[x]$ by the letter $k$, and sometimes 
write $k=\deg z$.

On the other hand, the \emph{grade} \cite{m4} of a polynomial $z \in Z[x]$ is any integer $g=\mathrm{grade}(z)$ satisfying $g \geq k$. 
The choice of the grade of a polynomial is arbitrary: nevertheless, some algebraic properties of polynomial matrices depend on the grade.

\begin{remark}
In some sense, the degree of a polynomial is an intrinsic property while the grade depends on its representation. 
In fact, informally speaking, the grade depends on how many zero coefficients one wishes to add in front of the polynomial.
\end{remark}

Let now $g$ be the grade of $z=\sum_{i=0}^g a_i x^i \in Z[x]$. The \emph{reversal} of $z$ with respect to its grade \cite{gkl, m4} is
\begin{equation}\label{rev}
\mathrm{Rev}_g z := \sum_{i=0}^g a_{g-i} x^i.
\end{equation}
The subscript $g$ will sometimes be omitted when the reversal is taken with respect to the degree of the polynomial, that is $\mathrm{Rev}_k z=:\mathrm{Rev} z$. 

Let now $\mathbb{F}$ be an arbitrary algebraically closed field. 

\begin{remark}
Although the hypothesis that $\mathbb{F}$ is algebraically closed is useful to state in a simpler way our results, it is not strictly necessary. See Section \ref{comments}.
\end{remark}

A well-known result that is crucial to us is that $\mathbb{F}[x]$ is guaranteed to be an Euclidean domain. 
Given $z_1, z_2 \in \mathbb{F}[x]$, not both zero, we denote by $\mathrm{GCD}(z_1,z_2)$ their greatest common divisor; we additionally require that $\mathrm{GCD}(z_1,z_2)$ is always monic so that it is uniquely defined. 
We say that $z_1$ and $z_2$ are \emph{coprime} if $\mathrm{GCD}(z_1,z_2)=1_{\mathbb{F}[x]}$.

Notice that a polynomial $z \in \mathbb{F}[x]$ can be thought of as a function 
$z(x):\mathbb{F}\rightarrow\mathbb{F}$. Thus, applying \eqref{rev}, in this case the formula $\mathrm{Rev}_gz(x)=x^g z(x^{-1})$ holds.

Let now $Z^{m \times p}$ be the set of $m \times p$ 
matrices with entries in $Z$; the case $p=1$ corresponds to the set of vector with 
$m$ elements in $Z$, denoted by $Z^m$. We are mainly interested in analysing $(\mathbb{F}[x])^{m \times p}$, 
the set of $m \times p$ \emph{polynomial matrices} 
with entries in $\mathbb{F}[x]$. 
$M_m(\mathbb{F}[x]):=(\mathbb{F}[x])^{m \times m}$ is the \emph{ring of square polynomial matrices of dimension $m$}. 
A square polynomial matrix $A \in M_m(\mathbb{F}[x])$ is said to be \emph{regular} if $\det A \neq 0_{\mathbb{F}[x]}$ and \emph{singular} otherwise. If $A$ is regular and $\det A \in \mathbb{F}$ then $A$ is called \emph{unimodular}.

\begin{remark}
Notice that $(\mathbb{F}[x])^{m \times p}=(\mathbb{F}^{m \times p})[x]$; or in other words, a polynomial matrix, defined as a matrix whose entries are polynomials, is also a matrix polynomial, defined as a polynomial whose coefficients are matrices. 
\end{remark}

The notions of grade and degree can be extended in a straightforward way to polynomial matrices, as follows: the grade (resp., the degree) of $A \in (Z[x])^{m \times p}$ is defined as $\max_{i,j}$grade$(A_{ij})$ (resp., as $\max_{i,j} \deg A_{ij}$). Analogously, the reversal of a polynomial matrix is defined just as in \eqref{rev}, after replacing $a_i \in Z$ with $B_i \in Z^{m \times p}$.

\subsection{Characteristic values, elementary divisors, and minimal indices}\label{completeigenstructure}
Let $A \in (F[x])^{m \times p}$, and let $\nu=:\min(m,p)$. Suppose that there exist $D_1, \dots, D_{\nu} \in \mathbb{F}[x]$ such that 
$A_{ij}=D_i \delta_{ij}$, where $\delta_{ij}$ is the Kronecker's delta. 
Then we write $A=\mathrm{diag}(D_{1},\ \dots, \ D_{\nu})$, and we say that $A$ is \emph{diagonal}. Notice that we use the notation 
indifferently for both square and rectangular polynomial matrices. 

The following Theorem, which in its most general is due to Frobenius \cite{frobenius}, is in point of fact valid for any 
matrix with entries in a principal ideal domain \cite{tom, bible}. 
\begin{theorem}\label{smiththeorem}
Let $P(x) \in (\mathbb{F}[x])^{m \times p}$. 
Then there exist two unimodular $A(x) \in M_m(\mathbb{F}[x])$ and $B(x) \in M_p(\mathbb{F}[x])$ such that 
\[S(x)=A(x)P(x)B(x)=\mathrm{diag}(d_1(x),\dots,d_{\nu}(x)),\] where $d_i(x) \in \mathbb{F}[x]$ is monic 
 $\forall i \leq \nu:=\min(m,p)$ and 
$d_{i}(x)|d_{i+1}(x) \ \ \forall i \leq \nu-1$.
\end{theorem} 

Such an $S(x) \in (\mathbb{F}[x])^{m \times p}$ is called the \emph{Smith form} \cite{bible, smith} of $P(x)$, 
and the $d_i(x)$ are called its \emph{invariant polynomials} \cite{tom, bible}. The Smith form, and thus the invariant 
polynomials, are uniquely determined by $P(x)$. Notice that a square polynomial matrix $P(x)$ is singular if and only if at least one 
of its invariant polynomials is zero.

Using the fact that $\mathbb{F}$ is algebraically closed, let us consider a factorization of the invariant polynomials over 
$\mathbb{F}[x]$: $d_i(x)=\prod_j (x-x_j)^{k_{j,(i)}}$.  
The factors $(x-x_j)^{k_{j,(i)}}$ are called the \emph{elementary divisors} of $P(x)$ \cite{tom, bible} corresponding to the 
\emph{characteristic value} $x_j$ \cite{tom}. Notice that, from Theorem \ref{smiththeorem},  
$i_1\leq i_2 \Rightarrow k_{j,(i_1)}\leq k_{j,(i_2)}$.

\begin{remark}
When $\mathbb{F}=\mathbb{C}$ the characteristic values of the polynomial matrix $P(x)$ are often 
called the eigenvalues of the matrix polynomial $P(x)$. Given an eigenvalue $x_0$, there is a Jordan chain of length $\ell$ at 
$x_0$ if and only if $(x-x_0)^{\ell}$ is an elementary divisor. The number of Jordan chains at $x_0$ is equal to the number of invariant 
polynomials that have $x_0$ as a root \cite{bible}.
\end{remark}

Let us now denote by $\mathbb{F}(x)$ the \emph{field of fractions} of the ring $\mathbb{F}[x]$.   
Let $\mathcal{V}$ be a vector subspace of $(\mathbb{F}(x))^p$, with $\dim \mathcal{V}=s$. Let $\{v_i\}$ be a polynomial basis 
for $\mathcal{V}$ with the property $\deg v_1 \leq \dots \leq \deg v_s$. Often we will arrange a polynomial basis in the matrix form 
$V(x)=[v_1(x),\dots,v_s(x)] \in (\mathbb{F}[x])^{p \times s}$. Clearly, polynomial bases always exist, because one may start from any 
basis with elements in the (vectorial) field of fractions, and then build a polynomial basis just by multiplying by the least common denominator. Let $\alpha_i:=\deg v_i$ be the degrees of the vectors of such a polynomial basis; the \emph{order} of $V(x)$ is defined \cite{forneyjr} as $\sum_{i=1}^s \alpha_i$. A polynomial basis is called \emph{minimal} \cite{forneyjr} if its order is minimal amongst all the polynomial bases for $\mathcal{V}$, and the $\alpha_i$ are called its \emph{minimal indices} \cite{forneyjr}. It is possible to prove \cite{forneyjr, tom} that, although there is not a unique minimal basis, the minimal indices are uniquely determined by $\mathcal{V}$.

The \emph{right minimal indices} \cite{dopico} of a polynomial matrix $P(x) \in (\mathbb{F}[x])^{m \times p}$ are defined as the minimal indices of $\ker P(x)$. Analogously, the \emph{left minimal indices} \cite{dopico} of $P(x)$ are the minimal indices of $\ker P(x)^T$.

Given the grade $g$ of $P(x)$, we say that $\infty$ is a characteristic value of $P(x)$ if $0_{\mathbb{F}}$ is a characteristic value of $Rev_g P(x)$. The elementary divisors corresponding to $\infty$ are defined \cite{diciotto} as the elementary divisors of $Rev_gP(x)$ corresponding to $0_{\mathbb{F}}$; if $x^{\ell}$ is an elementary divisor of $Rev_gP(x)$ we formally write that $(x-\infty)^{\ell}$ is an infinite elementary divisor of $P(x)$. Notice that the infinite elementary divisors of a polynomial matrix clearly depend on the arbitrary choice of its grade.

We complete this section with the following definition \cite{dopico}: the \emph{complete eigenstructure} of $P(x)$ is the set of both finite and infinite elementary divisors of $P(x)$ and of its left and right minimal indices.
\section{Rational transformations of polynomial matrices}\label{rationalfun}
Let $n(y), d(y) \in \mathbb{F}[y]$ be two nonzero, coprime polynomials. Let us define $N := \deg n(y)$, $D := \deg d(y)$, 
and $G:=\max(N,D)$. We will always suppose $G \geq 1$, that is $n(y)$ and $d(y)$ are not both elements of $\mathbb{F}$. We denote 
the coefficients of $n(y)$ and $d(y)$ as $n_i \in \mathbb{F}$, $i=0,\dots,N$ and $d_j \in \mathbb{F}$, $j=0,\dots,D$, that is 
$n(y)=\sum_{i=0}^N n_i y^i$, $d(y)=\sum_{i=0}^D d_i y^i$. 

Let us introduce the notation $\mathbb{F}^*:=\mathbb{F} \cup \{\infty\}$, having formally defined $\infty:=0_{\mathbb{F}}^{-1}$. 
We consider the generic rational function from $\mathbb{F}^*$ to $\mathbb{F}^*$:
\begin{equation}\label{rationalfunction}
x(y)=\frac{n(y)}{d(y)}.
\end{equation}
The function \eqref{rationalfunction} induces an application $\Phi_{g,n(y),d(y)}:(\mathbb{F}[x])^{m \times p} \rightarrow (\mathbb{F}[y])^{m \times p}$ defined as
\begin{equation}\label{transformation}
\Phi_{g,n(y),d(y)}(P(x)) = Q(y):=[d(y)]^g P(x(y))
\end{equation}
Here $g$ is the grade of $P(x) \in (\mathbb{F}[x])^{m \times p}$, so for any choice of $g$ a different application is defined. We will usually omit the functional dependence of $\Phi$ on $n(y)$ and $d(y)$ unless the context allows any possible ambiguity; also, if the grade is chosen to be $g=k$ we will sometimes omit the subscript $g$, that is $\Phi(P(x)):=\Phi_{k,n(y),d(y)}(P(x))$.

Since a polynomial matrix is also a matrix polynomial, we can write $P(x)=\sum_{i=0}^g P_i x^i$ for some $P_i \in \mathbb{F}^{m \times p}$, $i=0,\dots,g$. Notice that following the same point of view we can also write $Q(y)=\sum_{i=0}^{g} P_i [n(y)]^i [d(y)]^{g-i}$.

\begin{lemma}\label{degQ}
$\deg Q(y)=\deg \Phi_g(P(x))$ is less than or equal to $q:=$  
$gD+\max_{i : P_i \neq 0}(iN-iD)$. If $N \neq D$ the strict equality $\deg Q(y)=q$  always holds. Moreover, $q \leq gG$.
\end{lemma}

\begin{proof}
Writing $Q(y)$ as above, we can see it as a sum of the $k+1$ polynomial matrices $Q_i(y)= P_i [n(y)]^i [d(y)]^{g-i}$, $0\leq i \leq k$, with either $Q_i(y)=P_i=0$ or $\deg Q_i(y) = gD+i(N-D)$. Since the degree of the sum of two polynomials cannot exceed the greatest of the degrees of the considered polynomials, $\deg Q(y)$ cannot be greater than  $q$. Notice that if $N=G$ then $gG\geq q=kG+(g-k)D$ and the maximum is realised by $i=k$, while otherwise the maximum is realised by the smallest index $j$ such that $P_j \neq 0$, and $q=(g-j)G+jN$. This means that if $N<G$ and $P_0=0$ then $q<gG$, while $q=gG$ if $N<G$ but $P_0 \neq 0$.

Notice finally that, if $i_1 \neq i_2$, then $Q_{i_1}(y)$ and $Q_{i_2}(y)$ have the same degree if and only if $D=N$. Since 
$\deg Q_{i_1}(y) \neq \deg Q_{i_2}(y) \Rightarrow \deg (Q_{i_1}(y)+Q_{i_2}(y))=
\max(\deg Q_{i_1}(y),\deg Q_{i_2}(y))$, $D \neq N$ is a sufficient condition for $\deg Q(y) =q$.
\end{proof}

Lemma \ref{degQ} shows that 
$\deg Q(y) \leq q \leq gG$. The next Proposition describes the conditions under which the equality 
$\deg Q(y) = gG$ holds.

\begin{proposition}\label{degphipi}
Let $Q(y) =\Phi (P(x))$. It always holds $\deg Q(y) \leq gG$, and $\deg Q(y) < g G$ if and only if one of the following is true:
\begin{enumerate}
\item $N > D$ and $g>k$;
 \item $N \leq D$, and there exist a natural number $a \geq 1$ and a polynomial matrix   
$\hat{P}(x) \in (\mathbb{F}[x])^{m \times p}$ such that $P(x)=(x-\hat{x})^a \hat{P}(x)$, where 
$\hat{x} := n_G d_G^{-1}$ if $N=D=G$ and $\hat{x}:=0_{\mathbb{F}}$ if $N < D =G$.
\end{enumerate}
\end{proposition}

\begin{proof}
Lemma \ref{degQ} guarantees $\deg Q(y) \leq gG$. To complete the proof, there are three possible cases to be analysed.
\begin{itemize}
 \item If $G=N>D$, we know from Lemma \ref{degQ} that $\deg Q(y) = q$, and in this case $q=gD+kN-kD$. Therefore, 
$\deg Q(y)=gG \Leftrightarrow g=k$.
\item If $N=D=G$ and, we get $q=gG$. Let $Q(y)=\sum_{i=0}^{gG} \Theta_i y^i$: then, $\deg Q(y) < gG 
\Leftrightarrow \Theta_{gG} = 0_{(\mathbb{F}[x])^{m \times p}}$. On the other hand $\Theta_{gG}$ is the coefficient of $y^{gG}$ in 
$Q(y)=\sum_{i=0}^{g} P_i [n(y)]^i [d(y)]^{g-i}$, so $\Theta_{gG} = d_G^g \sum_{i=0}^g P_i n_G^i d_G^{-i} = d_G^g P(n_G d_G^{-1})$. 
Therefore, $\Theta_{gG}$ is zero if and only if every entry of 
$P(n_G d_G^{-1})$ is equal to $0_{\mathbb{F}[x]}$, or in other words if and only if 
$P(x) = (x-n_G d_G^{-1})^a \hat{P}(x)$ for some $a \geq 1$ and some suitable polynomial matrix $\hat{P}(x)$.
\item If $N<D=G$, recalling the proof of Lemma \ref{degQ} we conclude that $\deg Q(y) < gG$ if and only if $P_0 = 0$, which is 
equivalent to $P(x) = x^a \hat{P}(x)$ for a suitable value of $a \geq 1$ and some polynomial matrix $\hat{P}(x)$.
\end{itemize}
\end{proof}

The grade of $Q(y)$ is of course arbitrary, even though it must be greater than or equal to its degree. Since $\deg Q(y) \leq q \leq gG$, 
we shall define that the grade of $Q(y)$ is $gG$. This choice has an influence on the infinite elementary divisors of $Q(y)$, as they 
are equal to the elementary divisors corresponding to zero of the reversal of $Q(y)$ \emph{taken with respect to its grade}, that is 
$\mathrm{Rev}_{(gG)} Q(y)$. 

If one is interested in picking a different choice for the grade of $Q(y)$, the following Proposition explains how the infinite elementary divisors change.

\begin{proposition}
Let $P(x) \in (\mathbb{F}[x])^{m \times p}$, with $k=\deg P(x)$. Then the finite elementary divisors and the minimal indices of $P(x)$ do not depend on its grade, while the infinite elementary divisors do. Namely, let $\nu=\mathrm{min}(m,p)$; $x^{g-k} d_1(x),\dots,x^{g-k} d_{\nu}(x)$ are the invariant polynomials of $\mathrm{Rev}_g P(x)$ if and only if  $d_1(x),\dots,d_{\nu}(x)$ are the invariant polynomials of $\mathrm{Rev}_k P(x)$, for any choice of $g \geq k$.
\end{proposition}

\begin{proof}
Neither Theorem \ref{smiththeorem} nor the properties of $\ker P(x)$ and $\ker P^T(x)$ depend on the grade, so minimal indices and finite elementary divisors cannot be affected by different choices. Let $S(x)=A(x) \mathrm{Rev}_k P(x) B(x)$ be the Smith form of $\mathrm{Rev}_k P(x)$. We have $\mathrm{Rev}_g P(x)=x^{g-k} \mathrm{Rev}_k P(x)$, which implies that $x^{g-k}S(x)=A(x) \mathrm{Rev}_g P(x) B(x)$. Clearly $d_i(x)|d_j(x) \Leftrightarrow x^{g-k}d_i(x)|x^{g-k}d_j(x)$, and therefore we conclude that $x^{g-k}S(x)$ is the Smith form of $\mathrm{Rev}_g P(x)$.
\end{proof}

Let $\alpha,\beta,\gamma,\delta \in \mathbb{F}$. If $G=1$, $\Phi_{g,\alpha y+\beta,\gamma y+\delta}$ is clearly invertible and its inverse, with a little abuse of notation, is $\Phi_{g,\beta-\delta x,\gamma x-\alpha}(Q(y)) = [\gamma x - \alpha]^g Q(\frac{\beta-\delta x}{\gamma x - \alpha})$
. The most general case is analysed below.

\begin{proposition}\label{nobij}
Let us denote by $\mathbb{F}[x]_g$ the set of the univariate polynomials in $x$ whose degree is less than or equal to $g$.  Given $g,n(y),d(y)$, the application $\Phi_{g,n(y),d(y)}:(\mathbb{F}[x]_g)^{m \times p}\rightarrow (\mathbb{F}[y]_{(gG)})^{m \times p}$ is always an injective function, but it is not surjective unless $G = 1$.
\end{proposition}

\begin{proof}
Notice that $\Phi_g$ can be thought as acting componentwise, sending $P(x)_{ij}$ to $Q(y)_{ij}=\Phi_g(P(x)_{ij})$. Thus, it will be sufficient to show that, in the scalar case $\Phi_g : \mathbb{F}[x]_g\rightarrow\mathbb{F}[y]_{(gG)}$, $\Phi_g$ is surjective if and only if $G=1$. This is true because any polynomial that does \emph{not} belong to the set $R_y:=\{a(y)\in \mathbb{F}[y]:a(y)=\sum_{i=0}^{g} a_i [d(y)]^{g-i} [n(y)]^i\}$ cannot belong to the image of $\Phi_g$, and $R_y = \mathbb{F}[y]_{(gG)}$ if and only if $G = 1$. 
In fact, if we require that a generic $r \in \mathbb{F}[y]_{(gG)}$ belongs to $R_y$, we find out that the $g+1$ coefficients $a_i$ must satisfy $gG+1$ linear constraints.

To prove injectivity: $\Phi_g(P_1(x))=\Phi_g(P_2(x)) \Rightarrow P_1(x(y))=P_2(x(y)) \Rightarrow P_1(x)=P_2(x)$.
\end{proof}

Proposition \ref{nobij} tells us that, unless $G=1$ (the M\"{o}bius case), not every $Q(y)$ is such that $Q(y)=\Phi(P(x))$ for some $P(x)$.

A couple of additional definitions will turn out to be useful in the following. Let $x_0 \in \mathbb{F}^*$: we define $T_{x_0}$ as the counterimage of $x_0$ under the rational function $x(y)$. 

Moreover let $\alpha,\beta \in \mathbb{F}$ be such that $\frac{\alpha}{\beta}=x_0$ and $\alpha$ and $\beta$ are not both zero. For instance, we can pick 
$(\alpha,\beta)=(x_0,1_{\mathbb{F}})$ if $x_0 \neq \infty$ and $(\alpha,\beta)=(1_{\mathbb{F}},0_{\mathbb{F}})$ otherwise. Consider the polynomial equation
\begin{equation}\label{fundamental}
\alpha d(y)=\beta n(y).
\end{equation}
Let $S$ be the degree of the polynomial $\alpha d(y)-\beta n(y)$. Equation \eqref{fundamental} cannot have more than $S$ finite roots. 
If $S<G$ then we formally say $\infty \in T_{x_0}$. 

\begin{remark}\label{discu}
Notice that there are three cases that lead to $S<G$:
\begin{enumerate}
 \item $N=D=G$ and $x_0=n_G d_G^{-1}$, so that \eqref{fundamental} becomes $d_G n(y) = n_G d(y)$: 
in this case, $S$ is the maximum value of $i$ such that $n_i \neq x_0 d_i$;
\item $N<D=G$ and $x_0 = 0_{\mathbb{F}}$, so that \eqref{fundamental} becomes $n(y) = 0_{\mathbb{F}}$ and $S=N$;
\item $D<N=G$ and $x_0 = \infty$, so that \eqref{fundamental} is $d(y) = 0_{\mathbb{F}}$ and $S=D$.
\end{enumerate}
\end{remark}

We now define the multiplicity $m_0$ of any finite $y_0 \in T_{x_0}$ as the multiplicity of $y_0$ as a solution of the polynomial equation \eqref{fundamental}. If $\infty \in T_{x_0}$, its multiplicity is defined to be equal to $G-S$. Therefore, the sum of the multiplicities of all the (both finite and infinite) elements of $T_{x_0}$ is always equal to $G$, while the sum of the multiplicities of all the finite elements of $T_{x_0}$ is $S$.

The finite elements of $T_{x_0}$ are characterised by the following Proposition.

\begin{proposition}\label{crucial}
Let $y_0 \in \mathbb{F}$ and $x_0 \in \mathbb{F}^*$. Then $y_0 \in T_{x_0}$ if and only if $y_0$ is a solution of \eqref{fundamental} for $\alpha,\beta : x_0=\frac{\alpha}{\beta}$. 
Moreover, $\alpha_1 d(y_0)=\beta_1 n(y_0)$ and $\alpha_2 d(y_0)=\beta_2 n(y_0)$ if and only if $\frac{\alpha_1}{\beta_1}=\frac{\alpha_2}{\beta_2}$. 
\end{proposition}

\begin{proof}
The definition of $T_{x_0}$ implies the first part of the Proposition. The second part comes from the fact that $x(y)$ is a function.
\end{proof}

Proposition \ref{crucial}, albeit rather obvious, has the following important implication: 

\begin{corollary}\label{nocommonroot}
$x_0 \neq x_1 \Leftrightarrow T_{x_0} \cap T_{x_1} = \emptyset$.
Equivalently, $\alpha_1 \beta_2 \neq \alpha_2 \beta_1$ if and only if $[\beta_1 n(y) - \alpha_1 d(y)]$ and $[\beta_2 n(y) - \alpha_2 d(y)]$ $\in \mathbb{F}[y]$ are coprime. 

In particular, for any finite $x_0 \in \mathbb{F}$, $\Phi(x-x_0)$ and $d(y)$ are coprime.
\end{corollary}

In order to clarify the latter definitions, let us consider an example. 
Let $\mathbb{F}=\mathbb{C}$ and take $n(y)=y^4+y^3-y^2-y+1$, $d(y)=y^4$. 
$T_1$ is the set of the solutions of the equation $n(y)=d(y)$, so in this case $T_1=\{-1,1,\infty\}$. Moreover, the multiplicity of $-1$ and $1$ are, respectively, $1$ and $2$; 
since $S=3$ and $G=4$, the multiplicity of $\infty$ is by definition $G-S=1$. Within the same example, $T_{\infty}=\{0\}$; $0$ has multiplicity $4$ because it is a root of order $4$ of the equation $d(y)=0$.

\section{Main result}\label{themainresult}
We are now able to state our main Theorem.

\begin{theorem}\label{main}
Given $m,p \in \mathbb{N}_0$ and $n(y),d(y) \in \mathbb{F}[y]$, let $x_0 \in \mathbb{F}^*$ be a characteristic value of $P(x) \in (\mathbb{F}[x])^{m \times p}$, and let $(x-x_0)^{\ell_1}, \dots, (x-x_0)^{\ell_j}$ be the corresponding elementary divisors. Let $g$ be the grade of $P(x)$, define $G=\max(\deg n(y),\deg d(y))$ and let $gG$ be the grade of $Q(y)=\Phi_{g}(P(x)):=[d(y)]^g P(\frac{n(y)}{d(y)}) \in (\mathbb{F}[y])^{m \times p}$. Then for any $y_0 \in T_{x_0}$:
\begin{itemize}
\item $y_0$ is a characteristic value of $Q(y)$;
\item $(y-y_0)^{m_0 \ell_1}, \ \dots, \ (y-y_0)^{m_0 \ell_j}$ are elementary divisors corresponding to $y_0$ for $Q(y)$, where $m_0$ is the multiplicity of $y_0$.
\end{itemize}
Conversely, if $Q(y)=\Phi_g(P(x))$ for some $P(x)$, and if $y_0 \in \mathbb{F}^*$ is a characteristic value of $Q(y)$ with corresponding elementary divisors $(y-y_0)^{\kappa_1}, \ \dots, \ (y-y_0)^{\kappa_j}$:
\begin{itemize}
\item $x_0=\frac{n(y_0)}{d(y_0)}$ is a characteristic value of $P(x)$;
\item $m_0|\kappa_i$ $\forall i \leq j$, where $m_0$ is the multiplicity of $y_0$ as an element of $T_{x_0}$, and $(x-x_0)^{m_0^{-1} \kappa_1}, \ \dots, \ (x-x_0)^{m_0^{-1} \kappa_j}$ are elementary divisors corresponding to $x_0$ for $P(x)$.
\end{itemize}

In addition, the following properties hold: 
\begin{itemize}
\item the right minimal indices of $P(x)$ are $\beta_1, \ \dots, \ \beta_s$ if and only if the right minimal indices of $Q(y)$ are $G \beta_1, \ \dots, \ G \beta_s$;
\item the left minimal indices of $P(x)$ are $\gamma_1, \ \dots, \ \gamma_r$ if and only if the left minimal indices of $Q(y)$ are $G \gamma_1, \ \dots, \ G \gamma_r$.
\end{itemize}
\end{theorem}

For any choice of the application $\Phi_g$, Theorem \ref{main} gives a thorough description of the complete eigenstructure of $\Phi_g (P(x))$ with respect to the complete eigenstructure of $P(x)$. Notice that if $x(y)$ is a M\"{o}bius transformation then $m_0 \equiv 1$ and $G=1$, so the complete eigenstructure is unchanged but for the shift from one set of characteristic values to another. This is not the case for more general rational transformations, where other changes do happen.

The structure of the proof of Theorem \ref{main} is the following. First we prove the first part of the Theorem 
(the statement on elementary divisors). This is done dividing the statement in three cases:
\begin{enumerate}
\item[1.] $x_0 \in \mathbb{F}$ and $y_0 \in \mathbb{F}$;
\item[2.] $x_0 \in \mathbb{F}$ and $y_0 = \infty$;
\item[3.] $x_0 = \infty$.
\end{enumerate}
We first prove that the statement is true for case 1, 
then show that this implies that it is true for case 2. The validity of cases 1 and 2 implies case 3.

Finally, we prove the second part of the Theorem (the statement on minimal indices) with a constructive proof: we
 build a minimal basis of $Q(y)$ given a minimal basis of $P(x)$, and vice versa.
\section{Proof of Theorem \ref{main}: elementary divisors}\label{proof1}
The proof relies on the following Lemma, whose statement generalises \cite[Proposition 11.1]{bible}. 
The proof of the Lemma and more details are given in Section \ref{app}.

\begin{lemma}\label{lancaster}
Let $P(x) \in (\mathbb{F}[x])^{m \times p}$ and let $Q(x)=A(x)P(x)B(x)$ where $A(x)  \in M_m(\mathbb{F}[x])$ and 
$B(x) \in M_p(\mathbb{F}[x])$ are both regular, and 
suppose that $x_0 \in \mathbb{F}$ is neither a root of $\det A(x) \in \mathbb{F}[x]$ nor a root of 
$\det B(x) \in \mathbb{F}[x]$. 
Then  
$P(x)$ and $Q(x)$ have the same elementary divisors associated with $x_0$.
\end{lemma}

\subsection{Case 1}
Define $\nu:=\min(m,p)$, and 
let $P(x)=A(x)T(x)B(x)$ where $A(x)$ and $B(x)$ are unimodular polynomial matrices, 
$T(x)=:\mathrm{diag}(\delta_1(x),\dots,\delta_\nu(x))$ is the Smith form of $P(x)$, 
and $\delta_i(x)$ are its invariant polynomials.  
Let now $\hat{Q}(y):=\Phi(A(x))\Phi(T(x))\Phi(B(x))$. Clearly, $\hat{Q}(y)$ and $Q(y)$ differ only for a 
multiplicative factor $[d(y)]^{\lambda}$, $\lambda \in \mathbb{N}$; moreover, 
both $\det \Phi(A(x))$ and 
$\det \Phi(B(x))$ are nonzero whenever $d(y)\neq 0_{\mathbb{F}}$.
Notice that, if $x_0$ is finite, then for any $y_0 \in T_{x_0}$ there must hold $d(y_0)\neq 0_{\mathbb{F}}$ 
(Corollary \ref{nocommonroot}). Therefore, Lemma \ref{lancaster} implies that $Q(y)$, $\hat{Q}(y)$ and $S(y)$ have 
the same elementary divisors corresponding to $y_0$.

Unfortunately, $\Phi(T(x))$ may not be the Smith form of $\hat{Q}(y)$, because neither $\Phi(A(x))$ nor $\Phi(B(x))$ 
are necessarily unimodular and 
also because $\Phi(\delta_i(x))$ may not be monic. Nevertheless, it has the form 
$\mathrm{diag}([d(y)]^{k_1}\hat{\delta}_1(y), \dots,  [d(y)]^{k_{\nu}}\hat{\delta}_{\nu}(y))$, where 
$k_1 \geq k_2 \geq \dots \geq k_{\nu}$ and 
$\hat{\delta_i}(y):=\Phi(\delta_i(x))$. From Corollary \ref{nocommonroot}, $\hat{\delta}_i(y)$ and $d(y)$ cannot share common roots. 
To reduce $S(y)$ into a Smith form, we proceed by steps working on $2 \times 2$ principal submatrices.

In each step, we consider the submatrix $\left[\begin{smallmatrix} [d(y)]^{\gamma}
\hat{\delta}_i(y) & 0\\0 & 
[d(y)]^{\phi}\hat{\delta}_{j}(y)\end{smallmatrix}\right]$, where $\gamma:=k_i$ and $\phi:=k_j$, with $i<j$. If
$\gamma=\phi$, then do nothing; 
if $\gamma > \phi$, premultiply the submatrix by $\left[\begin{smallmatrix} 1_{\mathbb{F}} &
1_{\mathbb{F}}\\-b(y)q(y) & 1_{\mathbb{F}}-b(y)q(y)\end{smallmatrix}\right]$ 
and postmultiply it by $\left[\begin{smallmatrix} a(y) &
-q(y)\\b(y) & [d(y)]^{\gamma-\phi}\end{smallmatrix}\right]$, 
where $q(y)=\hat{\delta}_{j}(y)/\hat{\delta}_i(y)$ while $a(y)$
and $b(y)$ are such that $a(y) [d(y)]^{\gamma} 
\hat{\delta}_i(y) + b(y) [d(y)]^{\phi} \hat{\delta}_{j}(y) =
[d(y)]^{\phi} \hat{\delta}_i(y)$; the existence of two such 
polynomials is guaranteed by Bezout's lemma, since $[d(y)]^{\phi} \hat{\delta}_i(y)$ is
the greatest common divisor of $[d(y)]^{\gamma} \hat{\delta}_i(y)$ 
and $[d(y)]^{\phi}\hat{\delta}_{j}(y)$. It is easy to check that both matrices are
unimodular, and that the result of the matrix multiplications 
is $\left[\begin{smallmatrix} [d(y)]^{\phi}\hat{\delta}_i(y) & 0\\0 &
[d(y)]^{\gamma} \hat{\delta}_{j}(y)\end{smallmatrix}\right]$. 
Hence, by subsequent applications of this algorithm and after having defined a unimodular diagonal matrix 
$\Delta \in \mathbb{F}^{\nu \times \nu}$ chosen in such a way that the invariant polynomials of $S(y)$ are monic, it is possible 
to conclude that the Smith form of
$\Phi(T(x))$ is either $S(y):=\Delta \cdot \hat{S}(y)$ or $S(y):= \hat{S}(y)\cdot \Delta$ (whichever of the two products makes sense, 
depending on whether $m \leq p$ or not), where
$$\hat{S}(y)=\textrm{diag}(
[d(y)]^{k_m}\hat{\delta}_1(y),\dots,[d(y)]^{k_1}\hat{\delta}_m(y)).$$
Thus, the $i$th invariant polynomial of $P(x)$ has a root of multiplicity $\ell_i$ at $x_0$ if and only if the $i$th invariant 
polynomial of $\hat{Q}(y)$ has a root of multiplicity $m_0 \ell_i$ at $y_0 \in T_{x_0}$.
\subsection{Case 2}
By definition, the infinite elementary divisors for a given polynomial matrix are the elementary divisors corresponding 
to zero of the reversal of such polynomial matrix. Therefore, in order to prove Theorem \ref{main} for the case of $y_0=\infty$, we 
have to analyse the polynomial matrix $Z(y):=\mathrm{Rev}_{(gG)} Q(y) = y^{gG} [d(y^{-1})]^g P(x(y^{-1}))$, and find out what its 
relation to $P(x)$ is, with particular emphasis to its elementary divisors corresponding to $y_0=0_{\mathbb{F}}$.
Recalling Remark \ref{discu}, 
notice that there are two distinct subcases for which $\infty \in T_{x_0}$ for a finite 
$x_0 \in \mathbb{F}$. We will consider them separately.
\subsubsection{Subcase 2.1: $N=D=G$, $x_0=n_G d_G^{-1}$}
We get $x(y^{-1})=\frac{\mathrm{Rev} n(y)}{\mathrm{Rev} d(y)}$ and $y^G d(y^{-1})=\mathrm{Rev} d(y)$; therefore 
$Z(y)=[\mathrm{Rev} d(y)]^g P(\frac{\mathrm{Rev} n(y)}{\mathrm{Rev} d(y)})$. This means that we can prove analogous results for 
$Z(y)$ just by considering this time the new rational transformation $y \rightarrow x=\frac{\mathrm{Rev} n(y)}{\mathrm{Rev} d(y)}$. 
From Remark \ref{discu}, $0_{\mathbb{F}}$ is a root of multiplicity $G-S$ for the equation 
$\mathrm{Rev} n(y)=x_0 \mathrm{Rev} d(y)$; moreover, since we took the reversal with respect to the degree 
(or also because of Corollary \ref{nocommonroot}), $0_{\mathbb{F}}$ cannot be a root of $\mathrm{Rev}d(y)$.
Therefore, following the proof given above, one can state that $P(x)$ has $(x-x_0)^{\ell_1}, \dots, (x-x_0)^{\ell_j}$ as 
elementary divisors corresponding to $x_0$ if and only if $Z(y)$ has the $j$ elementary divisors 
$y^{(G-S) \ell_1}, \dots, y^{(G-S)\ell_j}$ corresponding to $0_{\mathbb{F}}$. The thesis follows immediately.
\subsubsection{Subcase 2.2: $N<D=G$, $x_0 = 0_{\mathbb{F}}$}
This time, we can write $x(y^{-1})=\frac{y^{G-N}\mathrm{Rev} n(y)}{\mathrm{Rev} d(y)}$ and $Z(y)=[\mathrm{Rev} d(y)]^g$   $P(\frac{y^{G-N}\mathrm{Rev} n(y)}{\mathrm{Rev} d(y)})$. It is therefore sufficient to consider the transformation $y \rightarrow x=y^{G-N}\frac{\mathrm{Rev} n(y)}{\mathrm{Rev} d(y)}$.

In fact, notice that $0_{\mathbb{F}}$ is a solution of multiplicity $G-N$ for the equation $y^{G-N}\mathrm{Rev} n(y)=0$ ($0_{\mathbb{F}}$ is neither a root of $\mathrm{Rev} n(y)$ nor a root of $\mathrm{Rev} d(y)$, because 
$\mathrm{Rev} n(0_{\mathbb{F}})=n_N \neq 0_{\mathbb{F}}$ and $\mathrm{Rev} d(0_{\mathbb{F}})=d_D \neq 0_{\mathbb{F}}$). Thus, $P(x)$ has the $j$ elementary divisors $x^\ell_1, \dots, x^\ell_j$ corresponding to $0_{\mathbb{F}}$ if and only if $Z(y)$ has the $j$ elementary divisors  $y^{(G-N) \ell_1}, \dots, y^{(G-N)\ell_j}$ corresponding to $0_{\mathbb{F}}$, and the thesis follows.
\subsection{Case 3}
By definition, the infinite elementary divisors of $P(x)$ are the elementary divisors corresponding to the characteristic value $0_{\mathbb{F}}$ for $R(x):=\mathrm{Rev}_g P(x)=x^g P(x^{-1})$. But let $\Psi_{g,n(y),d(y)}=\Phi_{g,d(y),n(y)}$ and $U(y)=\Psi_g(R(x))$, that is to say $U(y) = [n(y)]^g R(\frac{d(y)}{n(y)})$. A simple calculation gives 
\[U(y)=[n(y)]^g[\frac{d(y)}{n(y)}]^g P([\frac{d(y)}{n(y)}]^{-1})=[d(y)]^g P(\frac{n(y)}{d(y)})=\Phi_g(P(y))=Q(y).\] 

One can therefore follow the proof as in the previous Subsections, but starting from $R(x)$ and using a different transformation (notice 
that the equation $d(y)=0_{\mathbb{F}}$ defines both 
$T_{\infty}$ for the old transformation and $T_{0_{\mathbb{F}}}$ for the new transformation).

\section{Proof of Theorem \ref{main}: minimal indices}\label{proof2}
We shall only prove the theorem for right minimal indices. The proof for left minimal indices follows from the proof for right minimal indices and from the fact that $\Phi$ and the operation of transposition commute, that is $\Phi_g(P^T(x))=(\Phi_g(P(x)))^T$ $\forall$ $P(x) \in (\mathbb{F}[x])^{m \times p}$.
\subsection{$\Rightarrow$}
Let $\dim \ker P(x)=s$, and $V(x)=[v_1(x),\dots,v_s(x)]$ be a minimal basis for $\ker P(x)$, with minimal indices 
$\beta_i:=\deg v_i$ $\forall i=1,\dots,s$ and order $B:=\sum_{i=1}^s \beta_i$. For each value of $i$ let us define 
$w_i(y): = \Phi_{\beta_i}(v_i(x))$; we also define $W(y):=[w_1(y),\dots,w_s(y)]$. Clearly $\deg w_i(y) = G \beta_i$. 
Suppose in fact $\deg w_i(y) \neq G \beta_i$; applying Proposition \ref{degphipi} (in the case $g=k=\beta_i$), this would imply that there exists some $x_0 \in \mathbb{F}$ and some polynomial vector $u(x) \in (\mathbb{F}[x])^p$ such that $v_i(x)=(x-x_0) u(x)$. Hence, $[v_1(x),\dots,(x-x_0)^{-1}v_i(x),\dots,v_s(x)]$ would be a polynomial basis of order $B-1$ for $\ker P(x)$, leading to a contradiction. In order to prove that $W(y)$ is a minimal basis for $\ker Q(x)$ we must show that it is a basis and that it is minimal.

Clearly $w_i(y)$ lies in $\ker Q(y)$ for all $i$. In fact, $P(x) v_i (x) = 0$ implies that $Q(y) w_i (y) = 0$. 
So it is sufficient to show that $W(y)$, considered as an element of $(\mathbb{F}(x))^{p \times s}$, has rank $s$. Notice that 
$W(y)=V(x(y))\cdot\mathrm{diag}([d(y)]^{\beta_1},\dots,[d(y)]^{\beta_s})$. A well-known property of the rank is that, if 
$A_1=A_2 A_3$ and $A_3$ is square and regular, then $\mathrm{rk}(A_1)=\mathrm{rk}(A_2)$. Therefore 
$\mathrm{rk}(W(y))=\mathrm{rk}(V(x(y)$), because the diagonal matrix above is regular. Let $\hat{V}(x)$ be some regular 
$s \times s$ submatrix of $V(x)$, which exists because $\mathrm{rk}(V(x))=s$. 
By hypothesis, $\det(\hat{V}(x)) \neq 0_{\mathbb{F}[x]}$, which implies $\det(\hat{V}(x(y)))\neq 0_{\mathbb{F}(y)}$. Hence $s=\mathrm{rk}(V(x(y)))=\mathrm{rk}(W(y))$. Then $W(y)$ is a basis.

In order to prove that it is minimal, let us introduce the following lemma whose proof can be found in \cite{forneyjr}.

\begin{lemma}\label{forneyjr}
Let $\mathcal{V}$ be a vector subspace of $\mathbb{F}(x)^p$, with $\dim \mathcal{V} = s$. Let $H=[h_1,\dots,h_s]$ be a polynomial 
basis of order $A$ for $\mathcal{V}$ and define $\xi_i$, $i=1,\dots,\left(\begin{smallmatrix}
p\\
s\\
\end{smallmatrix}\right)$ to be the $s \times s$ minors (i.e. determinants of $s \times s$ submatrices) of $H$. Then the following statements are equivalent:
\begin{itemize}
\item $H$ is a minimal basis for $\mathcal{V}$
\item The following conditions are both true: (a) $\mathrm{GCD}(\xi_1, \dots, \xi_r)=1_{\mathbb{F}[x]}$ and (b) $\max_i$ $\deg \xi_i = A$.
\end{itemize}
\end{lemma}

So let $\xi_i(y)$ be the $s \times s$ minors of $W(y)$. We shall prove that (a) their GCD is $1_{\mathbb{F}[y]}$ and (b) their maximal degree is $G B = G \sum_{i=1}^s \beta_i$. By Lemma \ref{forneyjr}, these two conditions imply that $W(y)$ is minimal. Recall that $w_i(y) = \Phi_{\beta_i} (v_i(x))$, that is to say $w_i (y)=[d(y)]^{\beta_i} v_i(x(y))$. Any $s \times s$ submatrix of $W(y)$ is therefore obtained from the corresponding $s \times s$ submatrix of $V(x)$ by applying the substitution $x=x(y)$ and then multiplying the $i$th column by $[d(y)]^{\beta_i}$ for $i=1,\dots,s$. Let us call $\zeta_i(x)$ the $s \times s$ minors of $V(x)$. From the properties of determinants we obtain the relation $\xi_i (y)= \left(\prod_{i=1}^s [d(y)]^{\beta_i} \right) \zeta_i (x(y)) = [d(y)^B] \zeta_i (x(y)) = \Phi_B (\zeta_i (x))$.

Now for each $i$ let $\gamma_i := \deg \zeta_i (x)$ and $\delta_i :=\deg \xi_i (y) \leq \max_{j \leq \gamma_i}(N j-D j) + DB$ where the maximum is taken over those values of $j$ such that the $j$th coefficient of $\xi_i(y)$ is nonzero (Lemma \ref{degQ}). There are two cases. 
If $N \leq D=G$, $\delta_i \leq GB$, and applying Proposition \ref{degphipi} (with $g=B$), the inequality holds if and only if 
if $(x-\hat{x})|\zeta_i(x)$, where $\hat{x}=0_{\mathbb{F}}$ if $N<D$ and $\hat{x}=n_G d_G^{-1}$ if $N=D$;
notice that there must be at least one value of $i$ for which $\delta_i = GB$, otherwise $(x-\hat{x})$ would be a common factor of all 
the $\zeta_i(x)$, which is not possible because of Lemma \ref{forneyjr}. 
Finally, if $D<N=G$, $\delta_i = \gamma_i G + (B - \gamma_i) D$. Since $V(x)$ is minimal we have $\max_i(\gamma_i)=B$, 
which implies that also in this case $\max_i(\delta_i)=GB$. This proves condition (b).

Notice moreover that $\xi_i (y)=\Phi_B(\zeta_i(x))=[d(y)]^{B-\gamma_i}\Phi_{\gamma_i}(\zeta_i(x))$, where the first and the second factor are coprime (because of Corollary \ref{nocommonroot}). Let us prove the following Lemma.

\begin{lemma}\label{phiofgcd}
Let $p, q, r \in \mathbb{F}[x]$ with $r$ monic. Then, $\mathrm{GCD}_{\mathbb{F}[x]}(p,q)=r$ if and only if $\mathrm{GCD}_{\mathbb{F}[y]}(\Phi_{\deg p}(p),\Phi_{\deg q}(q))=\kappa \cdot \Phi_{\deg r}(r)$, where $\kappa \in \mathbb{F}$ is such that $\kappa \cdot \Phi_{\deg r}(r)$ is monic.
\end{lemma}

\begin{proof}
Let $\alpha, \beta$ be two suitable elements of $\mathbb{F}$ and let us write the prime factor decompositions 
$p=\alpha \cdot \prod (x - p_i)^{\pi_i}$, $q=\beta \cdot \prod (x - q_i)^{\theta_i}$, $r=\prod (x - r_i)^{\rho_i}$. Of course we have 
that $(x-r_i)^{\rho_i}|r$ if and only if $(x-p_i)^{\pi_i}|p$, $(x-q_i)^{\theta_i}|q$ and $\rho_i = \min (\pi_i,\theta_i)$.
We get $\Phi_{\deg p}(p) = \alpha \cdot \prod (n(y) - p_i d(y))^{\pi_i}$, 
$\Phi_{\deg q}(q) = \beta \cdot \prod (n(y) - q_i d(y))^{\theta_i}$ and 
$\Phi_{\deg r}(r) = \prod (n(y) - r_i d(y))^{\rho_i}$. The thesis follows by invoking Corollary \ref{nocommonroot}.
\end{proof}

Lemma \ref{phiofgcd} implies condition (a). This follows from the equation $\mathrm{GCD}_i(\xi_i(y))=\mathrm{GCD}_i([d(y)]^{B-\gamma_i}) \cdot \mathrm{GCD}_i (\Phi_{\gamma_i}(\zeta_i(x)))=1_{\mathbb{F}[y]} \cdot 1_{\mathbb{F}[y]}$, where the first $1_{\mathbb{F}[y]}$ comes from the fact that $\max_i(\gamma_i)=B$, while the second $1_{\mathbb{F}[y]}$ comes by applying the previous Lemma to $\mathrm{GCD}(\xi_1(y),\dots,\xi_s(y))=\mathrm{GCD}(\mathrm{GCD}(\dots\mathrm{GCD}(\xi_{2}(y),\xi_{1}(y))\dots))$ and from the identity $\Phi_0(1_{\mathbb{F}[x]})=1_{\mathbb{F}[y]}$.
\subsection{$\Leftarrow$}
To complete the proof, suppose now that $Q(y)=\Phi_g(P(x))$ for some $P(x) \in \mathbb{F}[x]$ and that $\hat{W}(y)$ is a minimal basis for $\ker Q(y)$, with minimal indices $\epsilon_1 \leq \dots \leq \epsilon_s$. The other implication that we proved in the previous subsection implies that $G|\epsilon_i$ $\forall$ $i$, so define $\beta_i=\frac{\epsilon_i}{G}$. Suppose that there exists a minimal basis $\hat{V}(x)=(\hat{v}_1(x),\dots,\hat{v}_s(x))$ for $\ker P(x)$; suppose moreover that an index $i_0 \in \{1,\dots,s\}$ exists such that $\deg \hat{v}_{i_0} \neq \beta_{i_0}$. Applying the reverse implication, this would imply that there is a minimal basis $\tilde{W}(y)=(\tilde{w}_1(y),\dots,\tilde{w}_s(y))$ for $\ker Q(y)$ whose $i_0$th right minimal index is not equal to $\epsilon_{i_0}$. This is absurd because every minimal basis has the same minimal indices.
\section{Extension to more relaxed hypotheses}\label{comments}
For the sake of convenience, we have so far assumed that the field $\mathbb{F}$ is algebraically closed. 
This unnecessary hypothesis can be dropped. To see it, assume that $\mathbb{F}$ is not algebraically closed and 
let $\mathbb{K}$ be the algebraic closure of $\mathbb{F}$. 
Then $(\mathbb{F}[x])^{m \times p} \subseteq (\mathbb{K}[x])^{m \times p}$, so we can use Theorem \ref{main} to identify the Smith forms 
of $P(x)$ and $Q(y)=\Phi_g(P(x))$ over the polynomial rings $\mathbb{K}[x]$ and $\mathbb{K}[y]$. We can then join back elementary 
divisors in $\mathbb{K}[x]$ and $\mathbb{K}[y]$ to form elementary divisors in $\mathbb{F}[x]$ and $\mathbb{F}[y]$. Of course, in this 
case an elementary divisor is no more necessarily associated with a characteristic value in $\mathbb{F}$. For instance, if 
$\mathbb{F}=\mathbb{Q}$, then the elementary divisor $x^2+2$ is not associated with any rational characteristic value, but if we consider 
the field of complex algebraic numbers $\mathbb{K}=\overline{\mathbb{Q}}$ then we can split it as $(x-\sqrt{2}i)(x+\sqrt{2}i)$ and 
associate it to the characteristic values $\pm \sqrt{2}i$. 
Similarly, the other results (e.g., Lemma \ref{phiofgcd}) that use the algebraic closure of $\mathbb{F}$ can be straightforwardly 
extended to a generic field 
$\mathbb{F}$ via an immersion into its algebraic closure $\mathbb{K}$.

\section{Proof of Lemma \ref{lancaster}}\label{app} 
Let $P(x) \in \mathbb{F}^{m \times p}[x]$. If $U(x):=[u_1(x), \ \dots, \ u_s(x)]$ is a minimal basis for $\ker P(x)$,
 we define $\ker_{x_0} P(x):=\mathrm{span} \left(\{u_1(x_0), \ \dots, \ u_s(x_0)\}\right) \subseteq \mathbb{F}^p$.
In general $\ker_{x_0} P(x)$ is a subset of $\ker P(x_0)$. It is 
a proper subset when $x_0$ is a characteristic value of $P(x)$, as is illustrated by the following example: let 
$\mathbb{F}=\mathbb{C}$ and
\[
P(x)=\left[\begin{array}{cccc}
                  x & 1 & 0 & 0\\
		  0 & x & 1 & 0\\
		  0 & 0 & 0 & 0\\
		  0 & 0 & 0 & x\\
                 \end{array}\right].
\]
Evaluating the 
polynomial at $0$, we get 
$\ker P(0)=\mathrm{span}(\{[1,0,0,0]^T,[0,0,0,1]^T\})$. 
On the other hand, a minimal basis for $\ker P(x)$ is $[1,-x,x^2,0]^T$, so $\ker_{0} P(x)=\mathrm{span}(\{[1,0,0,0]^T\})$.

We need now to slightly modify a definition given in \cite{bible} in order to extend it to
 the case of singular and/or rectangular polynomial
 matrices. 
A polynomial vector $v(x) \in (\mathbb{F}[x])^{p}$ is called a \emph{root polynomial} of order $\ell$ corresponding 
to $x_0$ for $P(x)$ if the following conditions are met:
\begin{enumerate}
\item $x_0$ is a zero of order $\ell$ for $P(x)v(x)$;
\item $v(x_0) \not\in \ker_{x_0} P(x)$.
\end{enumerate}

Observe that $v(x_0) \in \ker_{x_0} P(x) \Leftrightarrow
 \exists \ w(x) \in \ker P(x) \subseteq (\mathbb{F}(x))^p : w(x_0)=v(x_0)$. 
In fact, let $w(x)=U(x)c(x)$ for some $c(x) \in (\mathbb{F}(x))^s$ and $w(x_0)=v(x_0)$: 
then $v(x_0)=U(x_0)c(x_0) \in \ker_{x_0} P(x)$. Conversely, write $v(x_0)=U(x_0)c$ for some $c \in \mathbb{F}^s$ and notice that 
$U(x)c \in \ker P(x)$. Hence, condition 2. implies $v(x) \not\in \ker P(x)$.

In \cite[Proposition 1.11]{bible} it is shown that given three \emph{regular}
 polynomial matrices $P(x), A(x), B(x) \in M_n(x)$, and if $x_0$ is neither a root of $\det A(x)$ nor a root of $\det B(x)$, then $v(x)$ 
is a root polynomial of order $\ell$ for $A(x)P(x)B(x)$ corresponding to $x_0$ if and only  if $B(x)v(x)$ is a root polynomial of order
 $\ell$ corresponding to $x_0$ for $P(x)$. The next Proposition generalises this result.

\begin{proposition}\label{PandAPB}
 Let $P(x) \in \mathbb{F}^{m \times p}[x]$, $A(x) \in M_m(\mathbb{F}[x])$ and $B(x) \in M_p(\mathbb{F}[x])$. Suppose that both $A(x_0)$ 
and $B(x_0)$, with $x_0 \in \mathbb{F}$, are full rank matrices. Then $v(x)$ 
is a root polynomial of order $\ell$ corresponding to $x_0$ for $A(x)P(x)B(x)$ if and only  if $B(x)v(x)$ is a root polynomial of order
 $\ell$ corresponding to $x_0$ for $P(x)$.
\end{proposition}

\begin{proof}
Notice that if $A(x_0)$ and $B(x_0)$ are full rank then $A(x)$ and $B(x)$ are regular. 
In \cite{bible}, a root polynomial
 is defined for regular square polynomial matrices, so that condition 2. reduces to $v(x_0) \neq 0$. Nevertheless, the proof given 
in \cite[Proposition 1.11]{bible} for condition 1. does not actually use the regularity of $P(x)$, and it is therefore 
still valid when $P(x)$ is not a regular square polynomial matrix. To complete the proof: 
$v(x_0) \in \ker_{x_0} A(x)P(x)B(x) \Leftrightarrow \exists w_1(x) \in \ker A(x)P(x)B(x) : w_1(x_0)=v(x_0) \Leftrightarrow \exists w_2(x)
 \in \ker P(x) : w_2(x_0)=B(x_0)v(x_0) \Leftrightarrow B(x_0)v(x_0) \in \ker_{x_0} P(x)$. To build $w_2(x)$ from $w_1(x)$, simply put 
$w_2(x)=B(x)w_1(x)$ and use the fact that $A(x)$ is regular. To build $w_1(x)$ from $w_2(x)$, let $(B(x))^{-1}$ be the inverse matrix 
(over $\mathbb{F}(x)$) of $B(x)$, which exists because $B(x)$ is regular; then, put $w_1(x)=  (B(x))^{-1} w_2(x)$.
\end{proof}

Let $v_1(x),\dots,v_s(x)$ be root polynomials corresponding to $x_0$ of orders $\ell_1 \leq \dots \leq \ell_s$. 
We call them a \emph{maximal set of $x_0$-independent root polynomials} if:
\begin{enumerate}
 \item they are $x_0$-independent, i.e. $v_1(x_0),\dots,v_s(x_0)$ are 
linearly independent;
\item no $(s+1)$-uple of 
$x_0$-independent root polynomials corresponding to $x_0$ exists;
\item there are no root polynomials corresponding to $x_0$ of order $\ell>\ell_s$;
 \item for all $j=1,\dots,s-1$, there does not exist a root polynomial $\hat{v}_{j}(x)$ of order $\hat{\ell}_{j} > \ell_j$ such that 
$\hat{v}_j(x),v_{j+1}(x),\dots,v_s(x)$ are $x_0$-independent.
\end{enumerate}
As long as $\det B(x_0)$ and $\det A(x_0)$ are nonzero, it is easy to check that 
$v_1(x),\dots,v_s(x)$ are a maximal set of $x_0$-independent 
root polynomials for 
$A(x)P(x)B(x)$ if and only if $B(x)v_1(x),\dots,B(x)v_s(x)$ are 
a maximal set of $x_0$-independent 
root polynomials for $P(x)$.  
The next Proposition completes the proof of Lemma \ref{lancaster}.

\begin{proposition}\label{rootandsmith}
 $P(x) \in (\mathbb{F}[x])^{m \times p}$ has a maximal set of $x_0$-independent 
root polynomials, of order $\ell_1, \dots, \ell_s$, if 
and only
 if $(x-x_0)^{\ell_1},\dots,(x-x_0)^{\ell_s}$ are the elementary divisors of $P(x)$ associated with $x_0$.
\end{proposition}
\begin{proof}
Let $S(x)$ be the Smith form of $P(x)$, and recall that the inverse of a unimodular polynomial matrix
is still a unimodular  
polynomial matrix \cite{bible}. Thus, in view of Proposition \ref{PandAPB} and Theorem \ref{smiththeorem}, 
it suffices to prove the thesis for $S(x)$. 
If $S(x)$ is the zero matrix, 
it has neither a root polynomial nor an elementary divisor, so there is nothing to prove. Otherwise, let $\nu$ be the maximal value of 
$i$ such that $(S(x))_{ii}\neq 0_{\mathbb{F}[x]}$ and for $j=1,\dots,p$ let $e_j \in (\mathbb{F}[x])^p$ be the 
polynomial vector such 
that $(e_j)_i=\delta_{ij}$. If $\nu < p$, $[e_{\nu+1},\dots,e_p]$ is a minimal basis for $\ker S(x)$ and, 
being of order $0$, also 
for $\ker_{x_0} S(x)$. Suppose that $v_1(x),\dots,v_s(x)$ is a maximal set of $x_0$-independent root polynomials for $S(x)$.  
Let $k \leq 
\nu$ be the smallest index such that $(v_s(x_0))_k \neq 0_{\mathbb{F}}$: there must exist such an index because 
$v_s(x_0) \not\in \ker_{x_0} S(x)$.
Let $(S(x))_{kk}=(x-x_0)^{\mu} \theta(x)$, with $\theta(x_0)\neq 0_{\mathbb{F}}$. We get $(S(x)v_s(x))_k = 
(x-x_0)^{\mu} \theta(x) (v_s (x))_k$, so $\mu \geq \ell_s$. Actually, $\mu=\ell_s$, or $e_k$ would be a root 
polynomial of order greater than $\ell_s$, which is absurd.  
Then let $k'$ be the largest index not equal to $k$ and 
such that $(v_{s-1}(x_0))_{k'} \neq 0_{\mathbb{F}}$ (if such and index does not exist, then $v_{s-1}(x_0)$ 
is, up to a vector lying in $\ker_{x_0} S(x)$, a multiple of $e_k$ and thus 
$\ell_{s-1}=\ell_s$: in this case, replace without any loss of generality $v_{s-1}(x)$ by 
a suitable linear combination of $v_{s-1}(x)$ and $v_s(x)$). Following an argument similar as above, we can show that  
$(S(x))_{k'k'}=(x-x_0)^{\ell_{s-1}}\hat{\theta}(x)$, $\hat{\theta}(x_0) \neq 0_{\mathbb{F}}$. 
We repeat the process until we find all the $s$ sought elementary divisors. There cannot be more, 
otherwise $\dim \ker S(x_0) - \dim \ker_{x_0} S(x)>s$ and it would be possible 
to find an $(s+1)$-uple of $x_0$-independent root polynomials.  
Conversely, it is easy to check that $e_{\nu-s+1},\dots,e_{\nu}$ are a maximal set of $x_0$-independent root polynomials.
\end{proof}

\begin{remark}
Root polynomials carry all the information on Jordan chains \cite{bible}. Let 
$v(x)=\sum_{i=0}^{\ell-1}(x-x_0)^i v_i$, $v_i \in \mathbb{F}^m$, be a root polynomial of order $\ell$ corresponding to $x_0$ for $P(x)$. 
It is possible to prove that then $w(y)=\sum_{i=0}^{\ell-1} [d(y)]^{\ell-1-i}[n(y)-x_0 d(y)]^i v_i$ is a root polynomial of 
order $m_0 \ell$ corresponding to $y_0$ for $Q(y)$. The latter formula 
relates the Jordan chains of $Q(y)$ at $y_0$ to the Jordan chains of $P(x)$ at $x_0$.
\end{remark}

\section{Conclusions}

We have shown that if $P(x)$ and $Q(y)$ are polynomial matrices whose entries belong to the ring of univariate 
polynomials in $x$ (resp. $y$) with coefficients in any field, and if $P(x)$ and $Q(y)$ are related by a rational transformation $x(y)$, 
then the complete eigenstructures of $Q(y)$ and $P(x)$ are simply related.

\section{Acknowledgements}

I would like to thank N. Mackey for her talk \cite{iciam} on M\"{o}bius transformations that was one of the sources of inspiration for this generalisation on generic rational functions (the other one is \cite{dicksonpal}). 
I am also grateful to N. Mackey and D. S. Mackey for the subsequent discussion, and to D. A. Bini and L. Gemignani for useful comments.
Finally, I am indebted with an anonymous referee for valuable suggestions and remarks that helped improving the paper.


\begin{thebibliography}{10}
\bibitem{dopico}
F.~De~Ter\'{a}n, F.~Dopico and D.~S.~Mackey.
Linearizations of singular matrix polynomials and the recovery of minimal indices.
\newblock{\em Electron. J. Linear Algebra}, 18, 371--402, 2009.

\bibitem{forneyjr}
G.~D.~Forney~Jr.
Minimal bases of rational vector spaces, with applications to multivariable linear systems.
\newblock {\em SIAM J. Control}, 13(3):493--520, 1975.

\bibitem{frobenius}
G.~Frobenius.
 Theorie der linearen Formen mit ganzen Coefficienten.
\newblock {\em J. Reine Angew. Math. (Crelle)}, 86:146--208, 1878.

\bibitem{tom}
F.~R.~Gantmacher.
\newblock{\em The Theory of Matrices}.
\newblock AMS Chelsea, Providence, 1998.

\bibitem{dicksonpal}
L.~Gemignani and V.~Noferini.
The Ehrlich-Aberth method for palindromic matrix polynomials represented in the Dickson basis.
\newblock {\em Linear Algebra Appl.}, in press, doi:10.1016/j.laa.2011.10.035.

\bibitem{gkl}
I.~Gohberg, M.~A.~Kaashoek and P.~Lancaster.
General theory of regular matrix polynomials and band Toeplitz operators.
\newblock {\em Integral Equations Operator Theory}, 11(6):776--882, 1988.

\bibitem{bible}
I.~Gohberg, P.~Lancaster and L.~Rodman.
\newblock {\em Matrix Polynomials}.
\newblock Academic Press, New York, 1982.

\bibitem{golub}
G.~Golub and H.~A.~van~der~Vorst.
Eigenvalue computation in the 20th century.
\newblock {\em J. Comput. Appl. Math.}, 123, 35--65, 2000.

\bibitem{diciotto}
G.~E.~Hayton, A.~C.~Pugh and P. Fretwell.
Infinite elementary divisors of a matrix polynomial and implications.
\newblock {\em  Internat. J. Control}, 47:53--64, 1988.

\bibitem{m4}
D.~S.~Mackey, N.~Mackey, C.~Mehl and V.~Mehrmann.
Smith Forms of Palindromic Matrix Polynomials.
\newblock {\em Electron. J. Linear Algebra}, 22, 53--91, 2011.

\bibitem{iciam}
N.~Mackey.
Moebius Transformations of Matrix Polynomials.
\newblock{\em Talk given on Friday, July 22nd 2011}, 7th International Congress on Industrial and Applied Mathematics, Vancouver.

\bibitem{tisseur}
K.~Meerbergen and F.~Tisseur.
The quadratic eigenvalue problem.
\newblock {\em SIAM Rev.}, 43(2):235--286, 2001.

\bibitem{mehrmannvoss}
V. Mehrmann and H. Voss.
Nonlinear Eigenvalue Problems: A Challenge for Modern Eigenvalue Methods
{\em Mitt. Ges. Angew. Math. Mech.}, 27:121--151, 2005.

\bibitem{smith}
H.~J.~S.~Smith.
On Systems of Linear Indeterminate Equations and Congruences.
\newblock{\em Philos. Trans. R. Soc. Lond.}, 151:293--326, 1861.


\end{thebibliography}
\end{document}